\documentclass[11pt]{article}

\usepackage[a4paper,margin=1in]{geometry}
\usepackage{amsmath,amssymb,amsthm,mathtools}
\usepackage{microtype}
\usepackage{hyperref}
\usepackage{xcolor}
\usepackage{enumitem}
\usepackage{verbatim}
\usepackage{authblk}

\newtheorem{theorem}{Theorem}
\newtheorem{remark}[theorem]{Remark}
\newtheorem{lemma}[theorem]{Lemma}

\hypersetup{
  colorlinks=true,
  linkcolor=blue!50!black,
  citecolor=blue!50!black,
  urlcolor=blue!50!black
}

\title{Grokability in five inequalities}

\author{
Paata Ivanisvili\thanks{University of California, Irvine, pivanisv@uci.edu} \quad  Xinyuan Xie\thanks{University of California, Irvine, xinyuax7@uci.edu} }

\date{}

\begin{document}
\maketitle

\begin{abstract}
In this note, we report five mathematical discoveries made in collaboration with Grok, all of which have been subsequently verified by the authors. These include an improved lower bound on the maximal Gaussian perimeter of convex sets in $\mathbb{R}^n$, sharper $L_2$-$L_1$ moment comparison inequalities on the Hamming cube $\{-1,1\}^n$, a strengthened autoconvolution inequality, improved asymptotic bounds on the size of the largest $g$-Sidon sets in $\{1,\dots,n\}$, and an optimal balanced Szarek's inequality.
\end{abstract}

\section{Introduction}
In this note, we record several results in analysis, probability, convex geometry, and additive combinatorics, obtained through conversations with Grok. These results, to the best of our knowledge, have not previously appeared in the public literature.

The first result is an improved lower bound for the maximal Gaussian perimeter of convex sets—the first constant improvement since Nazarov's construction in 2003 \cite{Nazarov2003}. The second result is an improved upper and lower bound on the optimal exponential base for $L_2-L_1$ moment comparison inequalities on the Hamming cube. In particular, the improved lower bound answers a MathOverflow question posed by Noam Lifshitz in 2014 \cite{MO184286}. The third result is an optimal Khintchine-type inequality when the Rademacher random variables are conditioned on the middle slice of the Hamming cube, and we call it balanced Szarek's inequality. In this case, the random variables are no longer independent, but interestingly, the classical semigroup approach can still obtain the optimal constant. The fourth result is an improved lower bound on an autocorrelation inequality related to the size estimate of $g$-Sidon sets. This constant is recorded as the $C_{1a}$ constant in a curated collection of optimization constants in mathematics initiated by Tao \cite{Tao26} and maintained by Davis, Ivanisvili, and Tao \cite{OptConstC1a}. This improvement is of interest from the prospective of  building automatic reasoning systems with the assistance of LLMs, as will be explained in the next paragraph. These examples, together with recent works obtained in a similar fashion \cite{IvanisviliXie25,BubeckCoesterEldanGowersLeeLupsascaSawhneyScherrerSellkeSpearsUnutmazWeilYinZhivotovskiy25,FengTrinhBinghamKangZhangKimBarretoSchildkrautJungSeoPaganoChervonyiHwangHouGukovTsaiChoiJinLiWuShiuShihLeLuong26, WoodruffCohenAddadJainMaoZuoBateniBranzeiBrennerChenFengFortnowFuGuanHadizadehHajiaghayiJafariRavizJavanmardKarthikKawarabayashiKumarLattanziLeeLiPanageasPaparasPrzybockiSubercaseauxSvenssonTaherijamWuYogevZadimoghaddamZhouMatiasManyikaMirrokni26,JuGaoJiangWuSunChenYWangYWangZWangHePWuXiaoLiuDaiDong26}, suggest that the current frontier reasoning models may carry out sophisticated analytical arguments, explore the boundary of existing methods, and implement new ideas with a meaningful degree of accuracy. We hope this note helps the community stay informed about the growing power of these new tools in mathematical exploration and discovery.

 On the other hand, there has been a series of works on building automated reasoning systems with the assistance of LLMs \cite{Georgiev25,Wang25,Yuksekgonul26}. For each problem, these systems first parameterize it, encode the parametrization in computer code, for example in Python, and then iteratively refine (evolve) the code with the help of LLMs. The idea of building automated reasoning systems is very interesting, but one may hope for a more ambitious use of LLMs. For example, the $C_{1a}$ constant has been extensively studied using computer-assisted approaches, whether based on LLMs or on classical implementations, for searches in explicit and structured discrete spaces; see \cite{OptConstC1a} for detailed information about this constant. In particular, the best known lower bound was obtained by exhausting certain large but finite sets using about 20{,}000 CPU hours \cite{CloningerSteinerberger17}. 
In this note, we present a tiny improvement to the lower bound for the $C_{1a}$ constant, obtained within a few minutes of conversation with Grok, suggesting that explicit and structural searching leaves extra room that can be uncovered through LLM-based conversational approaches\footnote{The term ``conversational approach'' was taken from a post by Tao \cite{Tao25Mathstodon}. }. This example, together with the other examples presented in this note, suggests the potential for building reasoning systems that explore mathematical questions through natural language and mathematical arguments.


The rest of the paper is organized as follows. In the remainder of the introduction, we present concise statements of each problem and explain why they are interesting. The proofs are presented in Section~\ref{proofs}. In Appendix~\ref{appendix:A}, we provide links that fully disclose our corresponding conversations with Grok, which led to the new results for each question.



\subsection{Maximal Gaussian perimeter of convex sets}

For a convex set \(K\subseteq\mathbb{R}^n\), its Gaussian perimeter (or Gaussian surface area) is defined as
\[
\operatorname{GSA}(K)=\int_{\partial K}\varphi_n(x)\,d\sigma(x),
\qquad
\varphi_n(x):=(2\pi)^{-n/2}e^{-\|x\|^2/2}.
\]
The extremal quantity of interest is
\[
\Gamma(n):=\sup\{\operatorname{GSA}(K):K\subseteq\mathbb{R}^n \text{ is convex}\}.
\] The problem is to determine the growth of \(\Gamma(n)\) as \(n\to\infty\).

Ball\cite{Ball1993} proved that
\[
\frac{1}{e}\sqrt{\log n } \leq \Gamma(n)\le 4n^{1/4},
\]
Nazarov\cite{Nazarov2003} later showed that the order $n^{1/4}$ is sharp by proving
\[
e^{-5/4} < \liminf_{n\to\infty}\frac{\Gamma(n)}{n^{1/4}} \leq \limsup_{n\to\infty}\frac{\Gamma(n)}{n^{1/4}} < 0.64.
\]
The upper bound for $\Gamma_n$ was used to obtain sharper dimension dependence in multivariate Berry--Esseen bound over convex sets \cite{Bentkus03}. Rai\v{c} \cite{Raic2019} further refined the multivariate Berry--Essen bound and improved the upper bound of $\Gamma(n)$ to
\[
\Gamma(n)\le \sqrt{\frac{2}{\pi}}+0.59\bigl(n^{1/4}-1\bigr),
\] 
The problem is also relevant in learning theory. Klivans, O'Donnell, and Servedio showed that
Gaussian perimeter controls low-degree Hermite approximation, and hence the complexity of PAC
and agnostic learning under Gaussian distributions \cite{KOS2008}. More recently, Nadimpalli and Pascale\cite{NadimpalliPascale2025} revisited Nazarov's construction and recovered the Nazarov's lower-bound constant \(e^{-5/4}\) by an
alternative argument based on convex influence. Here we show that the constant in the lower bound can be improved further: 

\begin{theorem} \label{Gaussian} We have
    \[
\liminf_{n\to\infty}\frac{\Gamma_n}{n^{1/4}} \ \ge\ 0.31258.
\]
This is about $9\%$ improvement of the previous best known lower bound.  
\end{theorem}

In \cite{Nazarov2003}, the analytical argument is sophisticated, and several intermediate results are stated without proof. In Section~\ref{proofs}, we follow the same general strategy as in \cite{Nazarov2003}, while giving more detailed proofs and correcting a minor numerical inaccuracy in the original paper. We found that reasoning LLMs were able to work through these arguments, provide proofs for some statements left to the reader in \cite{Nazarov2003}, and fine-tune the parameters in Nazarov’s construction, leading to the present improvement.

\subsection{ $L_2$-$L_1$ moment comparison inequality}
For a function \(f:\{-1,1\}^n\to\mathbb{R}\) we write
\[
\|f\|_p:=\left(2^{-n}\sum_{x\in\{-1,1\}^n}|f(x)|^p\right)^{1/p},
\]
and we denote by \(\deg f\) its Fourier--Walsh degree, i.e.\ the degree of its multilinear expansion. We are interested in the optimal exponential base
\[
C_\ast:=\inf\Bigl\{C>0:\ \|f\|_2\le C^d\|f\|_1
\ \text{for every } d\ge 1,\ n\ge 1,\ \deg f\le d\Bigr\}.
\]

The problem of determining \(C_\ast\) was posed on MathOverflow by Noam Lifshitz in 2014, who asked in particular whether one can take \(C_\ast=\sqrt2\) \cite{MO184286}. The elementary example
\[
f(x)=\prod_{j=1}^d (1+x_j)
\]
already shows that \(C_\ast\ge \sqrt2\), since \(\|f\|_1=1\) and \(\|f\|_2=2^{d/2}\). The guess \(C_\ast=\sqrt2\) is also natural for structural reasons: for degree \(1\) it is exactly the sharp Khinchine inequality, due to Szarek \cite{Szarek76}; and if the same base \(\sqrt2\) were valid already for degree \(2\), then one would obtain the constant \(2\) for quadratic Rademacher chaoses, in accordance with the conjecture of Pe\l czy\'nski; see \cite[Remark~6]{IvanisviliTkocz}. On the upper-bound side, the real hypercontractivity yields \(C_\ast\le e\), see \cite[Theorem~9.22]{ODonnell}. A a more subtale argument via conformal maps, Hahn--Banach and Riesz representation theorem combined with complex hypercontractivity gives improvement $C_{*}\leq 2.69...$, see \cite{EskenazisIvanisvili}. We should also note that for \(d\)-homogeneous chaoses $h$ one has the sharper estimate
\[
\|h\|_2\le e^{d/2}\|h\|_1
\]
by Beckner's complex hypercontractivity \cite[Theorem~2]{IvanisviliTkocz}.

A useful observation is that if \(g\) is real-valued and \(f=g^2\), then
\[
\frac{\|f\|_2}{\|f\|_1}
=
\left(\frac{\|g\|_4}{\|g\|_2}\right)^2.
\]
Thus any family exhibiting large \(L^4/L^2\) growth immediately yields lower bounds for \(C_\ast\). The Gaussian model suggests taking \(g\) to be a Hermite polynomial evaluated on a long linear form. The next theorem shows that this already forces the base to be at least \(\sqrt3\).

\begin{theorem}\label{lem:sqrt3}
One has \(2.408...\geq C_\ast\ge \sqrt3\). 
\end{theorem}

We should note that the improvement in the upper bound from \(2.69\ldots\) to \(2.408\ldots\) was also known to the authors of \cite{EskenazisIvanisvili}. They did not revise the paper to reflect this, since the refinement follows verbatim from the argument in \cite{EskenazisIvanisvili}: one only needs to observe that the constant \(C\), defined in equation (191) there as the infimum of an explicit function, is in fact equal to \(2.408\ldots\).

On the lower-bound side, we note that the failure of the conjecture
\(C_\ast=\sqrt{2}\) in this generality was already known to some experts;
see, for instance, A. Samorodnitsky (private communication). Thus,
Theorem~\ref{lem:sqrt3} should not be regarded as a new mathematical fact
in itself. What is striking, however, is that reasoning models such as
Grok were able to recover both bounds, even though these observations do
not seem to have been recorded in the published literature.

\subsection{Optimal balanced Szarek's inequality}
The Khintchine inequality states that for independent Rademacher random variables
\(\varepsilon_1,\dots,\varepsilon_n\) and every \(p\geq 1\),
\[
A_p\Big(\sum_{i=1}^n a_i^2\Big)^{1/2}
\le
\Big(\mathbb{E}\Big|\sum_{i=1}^n a_i\varepsilon_i\Big|^p\Big)^{1/p}
\le
B_p\Big(\sum_{i=1}^n a_i^2\Big)^{1/2},
\]
where the constants \(A_p,B_p > 0\) depend only on \(p\). Determining the optimal constants in this inequality has attracted considerable attention: in
particular, the optimal constant when $p=1$ is due to Szarek \cite{Szarek76}, while the optimal constants for all $p$ were determined by Haagerup \cite{Haagerup81}.

A natural analogue is obtained by replacing the discrete cube \(\{-1,1\}^n\) with the middle
slice
\[
\Omega_n:=\Bigl\{x\in\{-1,1\}^n:\sum_{i=1}^n x_i=0\Bigr\},
\]
equipped with the uniform measure (here we need $n$ be even). Equivalently, one conditions the Rademacher random vector on the
event \(\sum_{i=1}^n x_i=0\) . In this model the coordinates are no longer independent, so the
classical proofs of Khintchine inequality may not directly apply. For \(p\ge 2\), such balanced Khintchine inequalities were studied by Spektor
\cite{Spektor16}. Later, Herscovici and Spektor \cite{HerscoviciSpektor20} computed the optimal constants when $p$ is even integer. By
contrast, the optimal constant for the endpoint \(p=1\) does not seem to have been documented in the literature. 

In the result below, we derive such balanced Khintchine-type iequality at the endpoint \(p=1\) with optimal constant, which may be viewed as a balanced analogue of Szarek's sharp \(L_1\) inequality.  Somewhat surprisingly, despite the loss of independence, the semigroup framework due to Kwapie\'n, Lata\l a, and Oleszkiewicz \cite{KwapienLatalaOleszkiewicz96,LatalaOleszkiewicz94,LatalaOleszkiewicz94} (see the second proof of Theorem 32 in \cite{NT23} for a modern presentation) still remains effective in obtaining the optimal constant. To deal with the lack of independence, we invoke a recent result of Filmus \cite{Filmus16} on orthogonal bases for the slice of the Hamming cube.

\begin{theorem}[Balanced Szarek's inequality]\label{Szarek}
Let \(n\in\mathbb{N}\) be even and \(n\ge 4\). Let \(x\) be distributed uniformly on
\[
\Omega_n:=\Bigl\{x\in\{-1,1\}^n:\sum_{i=1}^n x_i=0\Bigr\},
\]
and let \(a=(a_1,\dots,a_n)\in\mathbb{R}^n\). Then
\[
c_n \sqrt{\mathbb{E}\Big|\sum_{i=1}^n a_i x_i\Big|^2}
\le
\mathbb{E}\Big|\sum_{i=1}^n a_i x_i\Big|,
\]
where
\[
c_n=\sqrt{\frac{n-2}{2(n-1)}}
\]
is optimal.
\end{theorem}

\begin{remark}
The optimality is witnessed by the following examples, which is also the extremizer for classical Szarek's inequality. Let $a=(1, 1, 0,...,0)$. Then $|x_1+x_2|=2$ with probability $\frac{2\binom{n-2}{n/2 -2}}{\binom{n}{n/2}}=\frac{n-2}{2(n-1)}$, and $|x_1+x_2|=0$ otherwise. Thus, $\frac{\mathbb{E}|\sum_{i=1}^n a_i x_i|}{\sqrt{\mathbb{E}|\sum_{i=1}^n a_i x_i|^2}} = \sqrt{\frac{n-2}{2(n-1)}}$.
\end{remark}

\subsection{An autoconvolution inequality and $g$-Sidon sets}

Let an integrable $f : \mathbb{R} \to \mathbb{R}_{\ge 0}$ be supported on $[-1/4,1/4]$. We are interested in the largest universal constant $c>0$ such that
\[
\sup_{x\in\mathbb{R}} f*f(x)
\ge c \left( \int_{-1/4}^{1/4} f(x)\,dx \right)^2
\]
holds for every such function $f$. This constant \(c\) corresponds to the asymptotic size estimate for the  \(g\)-Sidon sets: a set \(A\subset\{1,2,\dots,n\}\) is called \(g\)-Sidon if every integer \(m\) has at most \(g\) representations of the form \(m=a+b\) with \(a,b\in A\). Denoting the largest size of a \(g\)-Sidon set for fixed $n$ and $g$ as
\[
\beta_g(n):=\max\{|A|:A\subset\{1,2,\dots,n\}\text{ is }g\text{-Sidon}\},
\]
\cite{MatolcsiVinuesa10} showed that the leading asymptotic constant for \(\beta_g(n)\) is a universal constant \(\sigma\) when $n$ and $g$ go to infinity, which can be reformulated in terms of the constant \(c\) as
\[
c=\frac{2}{\sigma^2}.
\]
Thus, studying the sharp autoconvolution inequality is equivalent to determining the leading asymptotic constant for the maximal size of \(g\)-Sidon sets. There have been many efforts devoted to bounding the constant \(c\) from above and below
\cite{SchinzelSchmidt02,MartinOBryant07,MartinOBryant09,MatolcsiVinuesa10,CloningerSteinerberger17,Georgiev25,Wang25,Yuksekgonul26}, where the best published bounds are
\[
1.28 \le c \le 1.50286.
\]

Recent progress on the upper bound \cite{Georgiev25,Wang25,Yuksekgonul26} is obtained through automatic reasoning systems assisted with LLMs, which iteratively modify Python codes describing candidate mathematical objects. At a high level, the reasoning LLMs in \cite{Georgiev25,Wang25,Yuksekgonul26} are used to explore the space formed by certain programming codes. One may naturally ask whether further improvements can still be achieved using existing techniques from the literature, and whether large language models can help uncover and exploit this remaining potential.  We provide positive evidence on this specific question.  In the following theorem, we provide an improved lower bound by refining an estimate in \cite{CloningerSteinerberger17}, which was suggested by LLMs.
\begin{theorem} \label{sidon}
 Let $f : \mathbb{R} \to \mathbb{R}_{\ge 0}$ be supported on $[-1/4,1/4]$. Then for every such function $f$, the following holds true
\[
\sup_{x\in\mathbb{R}} f*f(x)
\ge 1.2802 \left( \int_{-1/4}^{1/4} f(x)\,dx \right)^2.
\]   
\end{theorem}

The improvement can be obtained from Grok with a single natural prompt, simply by asking it to tighten an inequality in \cite{CloningerSteinerberger17}. The improvement is comparable to those obtained for the upper bound by LLM-based reasoning systems; see $C_{1a}$ constant \cite{OptConstC1a} for a neat record of recent progress on this question. From a mathematical perspective, however, this improvement is not particularly substantial, as it represents only a minor refinement of the estimate in \cite{CloningerSteinerberger17}. Moreover, we do not believe the authors missed anything essential, as their primary goal was to showcase the methodology. Nevertheless, we find this result impressive, since it shows that a purely conversational approach may extract a further improvement for a problem that has been extensively studied through structured and explicitly parameterized exhaustive search.



\section{Proofs}\label{proofs}

\subsection{Proof of Theorem~\ref{Gaussian}}
The proof presented in this section is the same as Nazarov's construction of a random convex set in \cite{Nazarov2003}. We show that a choice of specific parameters improves the previous lower bound. The random convex set is constructed in $\mathbb{R}^{n+1}$ as in Nazarov \cite{Nazarov2003} to avoid indexing $\varphi$, and we write $\Gamma_n:=\Gamma(n)$ for brevity.
\medskip

\begin{lemma}[Nazarov's exact expectation identity]\label{lem:ExpGSA}
Fix $n\ge 1$, $\rho>0$, and an integer $N\ge 1$.
Let $x_1,\dots,x_N$ be i.i.d.\ uniform on the unit sphere $S^n\subset\mathbb{R}^{n+1}$ and set
\[
Q:=\bigcap_{j=1}^N\{x\in\mathbb{R}^{n+1}:\ \langle x,x_j\rangle\le \rho\}.
\]
Then
\begin{equation}\label{eq:ExpGSA}
\mathbb{E}\,\operatorname{GSA}(Q)
=
N\,\frac{1}{\sqrt{2\pi}}e^{-\rho^2/2}
\int_{\mathbb{R}^n}\varphi_n(y)\,\bigl(1-p(|y|)\bigr)^{N-1}\,dy,
\end{equation}
where $p(r)$ depends only on $r\ge 0$ and is given by
\begin{equation}\label{eq:p_exact}
p(r)=
\Bigg(\int_{-\sqrt{r^2+\rho^2}}^{\sqrt{r^2+\rho^2}}
\Bigl(1-\frac{t^2}{r^2+\rho^2}\Bigr)^{\frac{n-2}{2}}\,dt\Bigg)^{-1}
\int_{\rho}^{\sqrt{r^2+\rho^2}}
\Bigl(1-\frac{t^2}{r^2+\rho^2}\Bigr)^{\frac{n-2}{2}}\,dt .
\end{equation}
\end{lemma}
\begin{remark}
In \cite{Nazarov2003}, the exponent of the integrand in the expression for $p(r)$ was given as $\frac{n-1}{2}$ instead of $\frac{n-2}{2}$ without proof, which is a minor inaccuracy that does not affect the final result. Here we include a proof for the reader's convenience.
\end{remark}
\begin{proof}
Since the laws of $x_j$'s are continuous and $x_j$ are independent, we have
$\mathbb{P}(x_i=x_j\text{ for some }i\neq j)=0$. Thus we may assume the normals
$x_1,\dots,x_N$ are pairwise distinct. For each $j$, let
\[
H_j:=\{x\in\mathbb{R}^{n+1}:\ \langle x,x_j\rangle=\rho\},
\qquad
F_j:=Q\cap H_j.
\]
Each $F_j$ is a (possibly empty) facet of $Q$ with outer unit normal $x_j$.
Moreover, $\partial Q$ is covered by the union of facets, and intersections
$F_i\cap F_j$ ($i\neq j$) have codimension at least $2$, hence $\sigma$-measure $0$.
Therefore,
\[
\operatorname{GSA}(Q)=\int_{\partial Q}\varphi_{n+1}(x)\,d\sigma(x)
=\sum_{j=1}^N\int_{F_j}\varphi_{n+1}(x)\,d\sigma(x).
\]
Taking expectation and using the fact that $\{x_i\}_{i \in [N]}$ has the same distribution,
\[
\mathbb{E}\,\operatorname{GSA}(Q)
=
N\cdot \mathbb{E}\!\left[\int_{F_1}\varphi_{n+1}(x)\,d\sigma(x)\right].
\]
By rotational invariance of $\varphi_{n+1}$, we may assume $x_1=e_{n+1}$. Therefore \[
\int_{F_1}\varphi_{n+1}(x)\,d\sigma(x) = \frac{1}{\sqrt{2\pi}}e^{-\rho^2/2}\int_{\mathbb{R}^n} \varphi_n(y) \mathbf{1}_{\{y:(y,\rho) \in F_1\}} dy,
\]
and by Tonelli, we have that 
\[
\mathbb{E}\!\left[\int_{F_1}\varphi_{n+1}(x)\,d\sigma(x)\right] =  \frac{1}{\sqrt{2\pi}}e^{-\rho^2/2}\int_{\mathbb{R}^n} \varphi_n(y) \mathbb{P}\big((y,\rho) \in F_1 \big) dy
\]
It remains to compute $\mathbb{P}\big((y,\rho) \in F_1 \big)$. For fixed $y\in\mathbb{R}^n$, the point $(y,\rho)$ belongs to $F_1$ iff it is not cut off
by any of the other half-spaces, i.e.
\[
\langle (y,\rho),x_j\rangle\le \rho,\qquad j=2,\dots,N.
\]
Let $r:=|y|$ and define
\[
p(r):=\mathbb{P}_{x\sim \mathrm{Unif}(S^n)}\big(\langle (y,\rho),x\rangle>\rho\big),
\]
which depends only on $|(y,\rho)|$ by rotation invariance of  $\mathrm{Unif}(S^n)$. Thus it only depends on $r$ for fixed $\rho$.
By independence,
\[
\mathbb{P}\big((y,\rho)\in F_1)=(1-p(r))^{N-1}.
\]
Hence
\[
\mathbb{E}\!\left[\int_{F_1}\varphi_{n+1}\,d\sigma\ \right]
=
\frac{1}{\sqrt{2\pi}}e^{-\rho^2/2}
\int_{\mathbb{R}^n}\varphi_n(y)\,(1-p(|y|))^{N-1}\,dy,
\]
and multiplying by $N$ gives \eqref{eq:ExpGSA}.

\emph{It remains to compute $p(r)$.} Fix $y$ with $|y|=r$ and put $u:=(y,\rho)\in\mathbb{R}^{n+1}$. For $x\sim \mathrm{Unif}(S^n)$, by rotational invariance, the random variable
$\frac{t}{|u|}:=\langle \frac{u}{|u|},x\rangle$ has the same distribution as $x_1$ as we may assume $\frac{u}{|u|}=e_1$. By \cite[Eq.~(8)]{BubeckDingEldanRacz16}, $x_1$ has density proportional to
\[
(1-|x_1|^2)^{(n-2)/2}.
\]
Therefore,
\[
p(r)=
\frac{\displaystyle\int_{\rho}^{|u|}
\left(1-\frac{t^2}{|u|^2}\right)^{(n-2)/2}\,dt}
{\displaystyle\int_{-|u|}^{|u|}
\left(1-\frac{t^2}{|u|^2}\right)^{(n-2)/2}\,dt}.
\]
Substituting $|u|=\sqrt{r^2+\rho^2}$ completes the proof. 
\end{proof}

\medskip

\begin{lemma}[Radialization]
 As in Nazarov \cite{Nazarov2003}, define
\begin{equation}\label{eq:f_c}
f(t):=t^{n-1}e^{-t^2/2},\qquad
c:=\Big(\int_0^\infty f(t)\,dt\Big)^{-1}.
\end{equation}
Then for every $n\geq 1$, every measurable $g:[0,\infty)\to[0,\infty]$,
\begin{equation}\label{eq:radialize}
\int_{\mathbb R^n}\varphi_n(y)\,g(|y|)\,dy
=
c\int_0^\infty f(r)\,g(r)\,dr.
\end{equation}
\end{lemma}
\begin{proof}
Let \(d\sigma\) denote the usual surface measure on \(S^{n-1}\). By polar coordinates and Tonelli,
\begin{align*}
\int_{\mathbb R^n}\varphi_n(y)g(|y|)\,dy
&=(2\pi)^{-n/2}\int_0^\infty\int_{S^{n-1}}
e^{-r^2/2}g(r)r^{n-1}\,d\sigma(\theta)\,dr \\
&=(2\pi)^{-n/2}|S^{n-1}|
\int_0^\infty r^{n-1}e^{-r^2/2}g(r)\,dr.
\end{align*}
Taking \(g\equiv 1\) gives
\[
1=(2\pi)^{-n/2}|S^{n-1}|\int_0^\infty f(r)\,dr,
\]
and hence
\[
(2\pi)^{-n/2}|S^{n-1}|=
\left(\int_0^\infty f(r)\,dr\right)^{-1}=c.
\]
Substituting this into the previous display gives the claim.
\end{proof}
\medskip

\begin{lemma}[Uniform upper bound for $p(r)$]\label{lem:p_bound}
Define the explicit Nazarov main term
\begin{equation}\label{eq:L_def}
L(n,\rho):=\frac{1}{\sqrt{2\pi}}\frac{1}{\rho}\exp\!\Big(\frac{\rho^4}{4n}\Big)e^{-\rho^2/2}.
\end{equation}
Fix constants $a>0$ and $W>0$. Then there exist constants $C_{a,W}>0$ and
$n_0(a,W)$ such that for all $n\ge n_0(a,W)$, all $\rho=a n^{1/4}$, and all
$r=\sqrt{n-1}+w$ with $|w|\le W$, one has
\begin{equation}\label{eq:p_bound}
p(r)\ \le\ L(n,\rho)\,\exp\!\Big(\frac{w\rho^2}{\sqrt{n}}\Big)\,
\exp\!\Big(\frac{C_{a,W}}{\sqrt{n}}\Big).
\end{equation}
\end{lemma}

\begin{proof}
Put $S:=r^2+\rho^2$. Then the denominator in \eqref{eq:p_exact} equals
\[
\sqrt S\int_{-1}^{1}(1-u^2)^{\frac{n-2}{2}}du
=
\sqrt S\,\sqrt\pi\,\frac{\Gamma(n/2)}{\Gamma((n+1)/2)}.
\]
By Wendel's inequality~\cite{Wendel1948},
\[
\frac{\Gamma(n/2)}{\Gamma((n+1)/2)}
\ge \sqrt{\frac{2}{n}},
\]
and hence
\[
\sqrt S\,\sqrt\pi\,\frac{\Gamma(n/2)}{\Gamma((n+1)/2)}
\ge \sqrt{\frac{2\pi S}{n}} .
\]
For the numerator, use
\[
\log(1-u)\le -u-\frac{u^2}{2},\qquad 0<u<1.
\]
Thus, for \(0<t<\sqrt S\),
\[
\Bigl(1-\frac{t^2}{S}\Bigr)^{\frac{n-2}{2}}
\le
\exp\!\left(
-\frac{n-2}{2S}t^2-\frac{n-2}{4S^2}t^4
\right).
\]
Since \(t\ge \rho\) on the interval of integration,
\[
\int_{\rho}^{\sqrt S}
\Bigl(1-\frac{t^2}{S}\Bigr)^{\frac{n-2}{2}}dt
\le
\exp\!\left(-\frac{n-2}{4S^2}\rho^4\right)
\int_{\rho}^{\infty}
\exp\!\left(-\frac{n-2}{2S}t^2\right)dt .
\]
Set 
\[
m(x):=\frac{1-\Phi(x)}{\varphi_1(x)}
=
\frac{\int_x^\infty e^{-u^2/2}\,du}{e^{-x^2/2}},
\qquad x>0.
\]
It follows from integration by parts that  \(m(x)\le x^{-1}\), i.e.,  equivalently,
\[
\int_x^\infty e^{-u^2/2}\,du\le \frac{1}{x}e^{-x^2/2}.
\qquad x>0,
\]
 Applying this with
\[
x=\rho\sqrt{\frac{n-2}{S}},
\]
we get, for \(n\ge3\),
\[
\int_{\rho}^{\infty}
\exp\!\left(-\frac{n-2}{2S}t^2\right)dt
=
\sqrt{\frac{S}{n-2}}
\int_{\rho\sqrt{(n-2)/S}}^\infty e^{-u^2/2}\,du
\le
\frac{S}{(n-2)\rho}
\exp\!\left(-\frac{n-2}{2S}\rho^2\right).
\]
Therefore
\[
\int_{\rho}^{\sqrt S}
\Bigl(1-\frac{t^2}{S}\Bigr)^{\frac{n-2}{2}}dt
\le
\frac{S}{(n-2)\rho}
\exp\!\Big(
-\frac{n-2}{2S}\rho^2-\frac{n-2}{4S^2}\rho^4
\Big).
\] Hence
\[
p(r)\le
\frac{1}{\sqrt{2\pi}\rho}\,
\frac{S}{n-2}\sqrt{\frac{n}{S}}\,
\exp\!\Big(
-\frac{n-2}{2S}\rho^2-\frac{n-2}{4S^2}\rho^4
\Big).
\]
Now set $\rho=a n^{1/4}$ and $r=\sqrt{n-1}+w$, with $|w|\le W$. Then
\[
S=r^2+\rho^2
=(\sqrt{n-1}+w)^2+\rho^2
=n+(2w+a^2)\sqrt n+O_{a,W}(1),
\]
uniformly for $|w|\le W$. Hence
\[
\frac{n-2}{S}
=
1-\frac{2w+a^2}{\sqrt n}+O_{a,W}(n^{-1}),
\qquad
\frac{n-2}{S^2}
=
\frac1n+O_{a,W}(n^{-3/2}).
\]
Since $\rho^2=a^2\sqrt n$ and $\rho^4=a^4 n$, we get
\[
-\frac{n-2}{2S}\rho^2
=
-\frac{\rho^2}{2}
+\frac{w\rho^2}{\sqrt n}
+\frac{\rho^4}{2n}
+O_{a,W}(n^{-1/2}),
\]
and
\[
-\frac{n-2}{4S^2}\rho^4
=
-\frac{\rho^4}{4n}
+O_{a,W}(n^{-1/2}).
\]
Adding the last two displays gives
\[
-\frac{n-2}{2S}\rho^2-\frac{n-2}{4S^2}\rho^4
=
-\frac{\rho^2}{2}
+\frac{w\rho^2}{\sqrt n}
+\frac{\rho^4}{4n}
+O_{a,W}(n^{-1/2}),
\]
uniformly for $|w|\le W$. Substituting the exponent expansion together with
\[
\frac{S}{n-2}\sqrt{\frac{n}{S}}
=
\exp\bigl(O_{a,W}(n^{-1/2})\bigr)
\]
gives
\[
p(r)\le
\frac{1}{\sqrt{2\pi}}\frac1\rho
\exp\!\Big(-\frac{\rho^2}{2}+\frac{\rho^4}{4n}\Big)
\exp\!\Big(\frac{w\rho^2}{\sqrt n}\Big)
\exp\!\Big(\frac{C_{a,W}}{\sqrt n}\Big),
\]
which is exactly \eqref{eq:p_bound} by the definition of $L(n,\rho)$. Increasing
$n_0(a,W)$ if necessary so that $n\ge3$ completes the proof.
\end{proof}

\medskip

\begin{lemma}[Uniform Laplace form for $c\,f(\sqrt{n-1}+w)$]\label{lem:cf}
Fix $W>0$. There exist constants $C_W>0$ and $n_1(W)$ such that for all
$n\ge n_1(W)$ and all $|w|\le W$,
\begin{equation}\label{eq:cf_asymp}
c\,f(\sqrt{n-1}+w)
=
\frac{1}{\sqrt{\pi}}e^{-w^2}\exp\!\Big(\frac{\theta_{n,W}(w)}{\sqrt{n}}\Big),
\qquad |\theta_{n,W}(w)|\le C_W.
\end{equation}
\end{lemma}

\begin{proof}
Let $t_0:=\sqrt{n-1}$ and $u:=w/t_0$. 
Then
\[
\log\frac{f(t_0+w)}{f(t_0)}=(n-1)\log(1+u)-t_0w-\frac{w^2}{2}.
\]
For $n\ge 4W^2+1$ we have $|u|\le 1/2$. Using $|\log(1+u)-(u-u^2/2)|\le 2|u|^3$ for $|u|\le 1/2$ and the identities
$(n-1)u=t_0w$, $(n-1)u^2=w^2$, we obtain
\[
\log\frac{f(t_0+w)}{f(t_0)}=-w^2+O_W(n^{-1/2})
\quad\text{uniformly for }|w|\le W.
\]
Next, $\int_0^\infty f(t)\,dt=2^{n/2-1}\Gamma(n/2)$, so $c f(t_0)$ can be
estimated by Stirling bounds for $\Gamma(n/2)$, giving
$c f(t_0)=\frac{1}{\sqrt{\pi}}\exp(O(n^{-1}))$.
Combining the two estimates yields \eqref{eq:cf_asymp}.
\end{proof}

\medskip

\begin{proof}[Proof of Theorem~\ref{Gaussian}]
Fix the rational parameters
\[
a:=\frac{6131}{5000}=1.2262,
\qquad
b:=\frac{2387}{1000}=2.387,
\qquad
W:=6.
\]
For each $n$, set
\[
\rho:=a n^{1/4},
\qquad
N:=\Big\lfloor \frac{b}{L(n,\rho)}\Big\rfloor,
\]
where $L(n,\rho)$ is defined in \eqref{eq:L_def}. Since $L(n,\rho)\to 0$
exponentially fast in $\sqrt{n}$, we have
\begin{equation}\label{eq:NL_to_b}
N\,L(n,\rho)\to b
\qquad\text{as }n\to\infty.
\end{equation}

Let $Q\subset\mathbb{R}^{n+1}$ be the random polytope from Lemma~\ref{lem:ExpGSA}.
By definition of $\Gamma(n+1)$,
\[
\Gamma(n+1)\ \ge\ \mathbb{E}\,\operatorname{GSA}(Q).
\]
Using Lemma~\ref{lem:ExpGSA} and radialization \eqref{eq:radialize},
\begin{equation}\label{eq:ExpGSA_radial}
\mathbb{E}\,\operatorname{GSA}(Q)
=
N\,\frac{1}{\sqrt{2\pi}}e^{-\rho^2/2}\,
c\int_0^\infty f(r)\bigl(1-p(r)\bigr)^{N-1}\,dr.
\end{equation}
Restrict to $r=\sqrt{n-1}+w$ with $|w|\le W$:
\begin{equation}\label{eq:restrictW}
\mathbb{E}\,\operatorname{GSA}(Q)
\ge
N\,\frac{1}{\sqrt{2\pi}}e^{-\rho^2/2}
\int_{-W}^{W} c\,f(\sqrt{n-1}+w)\,
\bigl(1-p(\sqrt{n-1}+w)\bigr)^{N-1}\,dw.
\end{equation}

By Lemma~\ref{lem:p_bound}, for $|w|\le W$ and large $n$,
\[
p(\sqrt{n-1}+w)\le q_n(w):=
L(n,\rho)\exp\!\Big(\frac{w\rho^2}{\sqrt{n}}\Big)\exp\!\Big(\frac{C_{a,W}}{\sqrt{n}}\Big).
\]
Since $q_n(w)\to 0$ uniformly on $[-W,W]$, for large $n$ we have $q_n(w)\le 1/2$ there,
and $\log(1-x)\ge -x-x^2$ on $[0,1/2]$ gives
\[
(1-q_n(w))^{N-1}\ge \exp\big(-(N-1)q_n(w)-(N-1)q_n(w)^2\big).
\]
Using $\rho^2/\sqrt{n}=a^2$ and \eqref{eq:NL_to_b}, we get pointwise for each fixed $w$,
\[
(N-1)q_n(w)\to b e^{a^2 w},
\qquad
(N-1)q_n(w)^2\to 0,
\]
hence
\begin{equation}\label{eq:lim_factor2}
\liminf_{n\to\infty}\bigl(1-p(\sqrt{n-1}+w)\bigr)^{N-1}
\ \ge\ \exp\big(-b e^{a^2 w}\big).
\end{equation}
Also Lemma~\ref{lem:cf} gives
\begin{equation}\label{eq:lim_factor1}
c\,f(\sqrt{n-1}+w)\to \frac{1}{\sqrt{\pi}}e^{-w^2} \qquad\text{as }n\to\infty.
\end{equation}

Next, by the definition \eqref{eq:L_def} we have the exact identity
\[
\frac{1}{\sqrt{2\pi}}e^{-\rho^2/2}=\rho e^{-\rho^4/(4n)}\,L(n,\rho).
\]
Therefore, using \eqref{eq:NL_to_b} and $\rho=a n^{1/4}$,
\begin{equation}\label{eq:prefactor_limit}
\frac{N}{n^{1/4}}\frac{1}{\sqrt{2\pi}}e^{-\rho^2/2}
=
\frac{N L(n,\rho)}{n^{1/4}}\rho e^{-\rho^4/(4n)}
\to
b\,a\,e^{-a^4/4} \qquad\text{as }n\to\infty.
\end{equation}

Apply Fatou's lemma to \eqref{eq:restrictW} and use
\eqref{eq:lim_factor2}, \eqref{eq:lim_factor1}, \eqref{eq:prefactor_limit}. We obtain
\begin{equation}\label{eq:liminf_W}
\liminf_{n\to\infty}\frac{\Gamma(n+1)}{n^{1/4}}
\ge
b\,a\,e^{-a^4/4}\cdot \frac{1}{\sqrt{\pi}}
\int_{-W}^{W}\exp\big(-w^2-b e^{a^2 w}\big)\,dw.
\end{equation}

\smallskip\noindent
\emph{Certified numerical lower bound (with $W=6$).}
Let
\[
F(w):=\exp\big(-w^2-b e^{a^2 w}\big),
\qquad
J:=\int_{-6}^{6}F(w)\,dw.
\]
We lower bound $J$ by the composite trapezoidal rule with step $h:=5\cdot 10^{-4}$.
Let $N_h:=24000$ and $w_k:=-6+kh$ ($0\le k\le N_h$). Define
\[
T:=h\Big(\tfrac12 F(w_0)+\sum_{k=1}^{N_h-1}F(w_k)+\tfrac12 F(w_{N_h})\Big).
\]
By the standard error estimate for the composite trapezoidal rule
\cite[Sec.~4.4, Thm.~4.5]{BurdenFairesBurden2016}, if
\(f\in C^2([A,B])\), \(h=(B-A)/m\), and
\[
T_h
=
h\left(\frac{f(A)+f(B)}{2}
+\sum_{k=1}^{m-1} f(A+kh)\right),
\]
then there exists \(\xi\in(A,B)\) such that
\[
\int_A^B f(x)\,dx-T_h
=
-\frac{B-A}{12}h^2 f''(\xi).
\]
In particular,
\[
\left|\int_A^B f(x)\,dx-T_h\right|
\le
\frac{B-A}{12}h^2\|f''\|_{L^\infty[A,B]}.
\]
Applying this with \(A=-6\), \(B=6\), \(f=F\), and \(m=24000\), gives
\[
|J-T|\le \frac{12}{12}h^2\max_{[-6,6]}|F''|
=h^2\max_{[-6,6]}|F''|.
\]
Writing $F=e^{g}$ with $g(w)=-w^2-b e^{a^2 w}$, one has
$F''=(g''+(g')^2)F$. A termwise global estimate using
$\sup_{t>0}t e^{-bt}=1/(be)$,
$\sup_{t>0}t^2 e^{-bt}=(2/b)^2e^{-2}$,
$\sup_{w\in\mathbb{R}}w^2e^{-w^2}=1/e$,
$\sup_{w\in\mathbb{R}}|w|e^{-w^2}=1/\sqrt{2e}$
yields, for all $w\in\mathbb{R}$,
\[
|F''(w)|\le
M(a):=
2+\frac{a^4}{e}+\frac{4}{e}+\frac{4a^2}{e\sqrt{2e}}+4a^4 e^{-2}.
\]
For $a=6131/5000$ one checks $M(a)<6.476$.

A direct finite computation gives
\[
T \geq 0.33310717054594
\]
and therefore
\[
J\ge T - M(a)h^2 \ \ge\ 0.33310555154594.
\]
Moreover,
\[
b\,a\,e^{-a^4/4} \geq 1.6632596039 \qquad
\frac{1}{\sqrt{\pi}}\ \ge\ 0.5641895835.
\]
Hence the right-hand side of \eqref{eq:liminf_W} with $W=6$ is at least
\[
1.6632596039 \times 0.5641895835 \times 0.33310555154594 >\ 0.312584.
\]
Finally, since $(n/(n+1))^{1/4}\to 1$, this implies
\[
\liminf_{n\to\infty}\frac{\Gamma(n)}{n^{1/4}}\ \ge\ 0.312584,
\]
which is exactly Theorem~\ref{Gaussian}.
\end{proof}

\subsection{Proof of Theorem~\ref{lem:sqrt3}}

\begin{proof}[Proof of Theorem~\ref{lem:sqrt3}]
Let \(\gamma\) be the standard Gaussian measure on \(\mathbb{R}\), and let \(\mathrm{He}_m\) denote the monic probabilists' Hermite polynomial of degree \(m\), normalized so that
\[
\int_{\mathbb{R}} \mathrm{He}_m(x)\mathrm{He}_k(x)\,d\gamma(x)=m!\,\delta_{mk}.
\]
In particular,
\[
\|\mathrm{He}_m\|_{L^2(\gamma)}^2=m!.
\]

For \(N\ge 1\), define
\[
S_N(x):=\frac{x_1+\cdots+x_N}{\sqrt N},
\qquad x\in\{-1,1\}^N,
\]
and
\[
F_{m,N}(x):=\mathrm{He}_m(S_N(x))^2.
\]
Since \(\mathrm{He}_m\) has degree \(m\), the function \(\mathrm{He}_m(S_N(x))\) is a polynomial of total degree at most \(m\) in the coordinates \(x_1,\dots,x_N\). After multilinear reduction on the cube (using \(x_i^2=1\)), its Fourier--Walsh degree is still at most \(m\). Hence
\[
\deg F_{m,N}\le 2m.
\]

If \(\varepsilon_1,\dots,\varepsilon_N\) are i.i.d.\ Rademacher signs, then
\[
\|F_{m,N}\|_1=\mathbb{E}\,\mathrm{He}_m\!\left(\frac{\varepsilon_1+\cdots+\varepsilon_N}{\sqrt N}\right)^2,
\]
and
\[
\|F_{m,N}\|_2
=
\left(
\mathbb{E}\,\mathrm{He}_m\!\left(\frac{\varepsilon_1+\cdots+\varepsilon_N}{\sqrt N}\right)^4
\right)^{1/2}.
\]
Now
\[
\frac{\varepsilon_1+\cdots+\varepsilon_N}{\sqrt N}\ \Longrightarrow\ G
\qquad(N\to\infty),
\]
where \(G\sim N(0,1)\), and in fact all moments converge. Since \(\mathrm{He}_m^2\) and \(\mathrm{He}_m^4\) are fixed polynomials, it follows that
\[
\lim_{N\to\infty}\frac{\|F_{m,N}\|_2}{\|F_{m,N}\|_1}
=
\frac{\|\mathrm{He}_m\|_{L^4(\gamma)}^2}{\|\mathrm{He}_m\|_{L^2(\gamma)}^2}.
\]

We now use the asymptotics of Larsson-Cohn. By \cite[Theorem~2.1(b)]{LarssonCohn}, for every fixed \(p>2\),
\[
\|\mathrm{He}_m\|_{L^p(\gamma)}
=
c(p)\,m^{-1/4}\sqrt{m!}\,(p-1)^{m/2}(1+o(1))
\qquad(m\to\infty),
\]
where \(c(p)>0\) is an explicit constant. Taking \(p=4\), we obtain
\[
\|\mathrm{He}_m\|_{L^4(\gamma)}
=
c(4)\,m^{-1/4}\sqrt{m!}\,3^{m/2}(1+o(1)),
\]
and therefore
\[
\frac{\|\mathrm{He}_m\|_{L^4(\gamma)}^2}{\|\mathrm{He}_m\|_{L^2(\gamma)}^2}
=
c(4)^2\,m^{-1/2}\,3^m(1+o(1)).
\]
Consequently,
\[
\lim_{m\to\infty}
\left(
\frac{\|\mathrm{He}_m\|_{L^4(\gamma)}^2}{\|\mathrm{He}_m\|_{L^2(\gamma)}^2}
\right)^{1/(2m)}
=
\sqrt3.
\]

Let \(C<\sqrt3\). By the previous limit, one can choose \(m\) so large that
\[
\frac{\|\mathrm{He}_m\|_{L^4(\gamma)}^2}{\|\mathrm{He}_m\|_{L^2(\gamma)}^2}>C^{2m}.
\]
Fix such an \(m\). Then, by the convergence in \(N\), for all sufficiently large \(N\),
\[
\frac{\|F_{m,N}\|_2}{\|F_{m,N}\|_1}>C^{2m}.
\]
But \(\deg F_{m,N}\le 2m\), so the inequality
\[
\|f\|_2\le C^d\|f\|_1
\]
fails for \(d=2m\). Since this happens for every \(C<\sqrt3\), we conclude that \(C_\ast\ge \sqrt3\).
\end{proof}

\subsection{Proof of Theorem~\ref{Szarek}}
Let $n \in \mathbb{N}$ be even. In this section, we show that the balanced Szarek's inequality, i.e.,  $(x_1,...,x_n)$ is uniformly distributed on $\Omega_n:=\{(x_1,...,x_n) \in \{-1,1\}^n:\sum_{i=1}^n x_i =0\}$, holds with the optimal constant.

The derivation of the optimal constant follows the semigroup proof framework for Szarek's inequality (also called the \(L_1\) Khintchine inequality) due to Kwapie\'n, Lata\l a, and Oleszkiewicz \cite{KwapienLatalaOleszkiewicz96,LatalaOleszkiewicz94,LatalaOleszkiewicz94} (see the second proof of Theorem 32 in \cite{NT23} for a modern presentation). In their setting, the random variables are independent, whereas in our setting they are slightly dependent. Their framework, however, still works. To deal with the lack of independence, we invoke a recent result of Filmus \cite{Filmus16} on orthogonal bases for the slice of the Hamming cube.

This may suggest that a certain symmetry of the distribution can be a suitable replacement for the independence condition in some Khintchine-type inequalities. 

\subsubsection{Harmonic Multilinear Polynomial}
Let $f:\mathbb{R}^n \to \mathbb{R}$ be a multilinear polynomial. We call $f$ a harmonic multilinear polynomial if 
\[
\sum_{i=1}^n \frac{\partial f}{\partial x_i}(x)=0 \quad \text{for all } x \in \mathbb{R}^n.
\]
One can check that the harmonic multilinear polynomials on $\mathbb{R}^n$ form a vector space, which we denote by $H_n$. Note that $H_n$ has dimension $\binom{n}{\lfloor\frac{n}{2}\rfloor}$ and every element in $H_n$ is a multilinear polynomial of degree at most $\frac{n}{2}$; see \cite[Lemma 2]{Filmus16} for a proof. 

A probability distribution $\mu$ on $\mathbb{R}^n$ is called exchangeable if it is invariant under permutations of the coordinates. Given an exchangeable distribution $\mu$ with mild regular conditions, we define the inner product on $H_n$ by 
\[
\langle f,g \rangle=\mathbb{E}_{x\sim\mu}[f(x)g(x)], 
\]
and set $\| f \|_2^2 = \langle f,f \rangle$. 
Filmus \cite[Section 3]{Filmus16} constructed a finite collection of functions $\{\chi_B\}_{B \in B_n}$ that forms an orthogonal basis of $H_n$ with respect to any exchangeable probability distribution $\mu$, and showed that every $f\in H_n$ admits the following Young-Fourier expansion  (whenever $\|\chi_B\|_2^2 > 0$ for all $B \in B_n$):
\[
f = \sum_{B \in B_n}\widehat f(B)\chi_B,
\]
where $\widehat f(B)=\frac{\langle f,\chi_B \rangle}{\|\chi_B\|_2^2}$. We will use the following facts from \cite[Section 3]{Filmus16}: every $\chi_B$ in the basis is a homogeneous multilinear polynomial indexed by $B$, where $B$ is a list of increasing numbers in $[n]$; the length of $B$, denoted by $|B|$, equals the degree of $\chi_B$.  

Filmus \cite[Section 6]{Filmus16} further defined a Laplacian operator $L$ by 
\[
Lf = f - \frac{1}{\binom{n}{2}}\sum_{1\leq i<j \leq n}f^{(i,j)},
\]
where $f^{(i,j)}=f(x^{(i,j)})$ and $x^{(i,j)}$ is obtained from $x$ by swapping its $i$th and $j$th coordinates. He also showed that $L$ admits the following spectral decomposition:
\[
Lf = \sum_{B \in B_n} \frac{2|B|(n+1-|B|)}{n(n-1)}\widehat f(B) \chi_B.
\]

\subsubsection{Balanced Szarek's inequality with best constant}
\begin{theorem}\label{uniqueness}
Let $n \in \mathbb{N}$ be even. For any function $f:\Omega_n \to \mathbb{R}$, there exists a unique $\tilde f \in H_n$ such that $\tilde f(x)= f(x)$ for all $x \in \Omega_n$.
\end{theorem}
\begin{proof}
Since the dimension of functions on $\Omega_n$ is $\binom{n}{\frac{n}{2}}$, which matches the cardinality of $\{\chi_B\}_{B\in B_n}$, it suffices to show that $\{\chi_B\}_{B\in B_n}$ is a set of linearly independent vectors when restricted to $\Omega_n$. Equip $\mathbb{R}^n$ with the uniform probability measure on $\Omega_n$. Thus it suffices to show that $\{\chi_B\}_{B\in B_n}$ forms an orthogonal basis with respect to the uniform probability measure on $\Omega_n$. Since the uniform probability measure on $\Omega_n$ is exchangeable, by Theorem 9 in \cite{Filmus16}, we know that $\langle \chi_A,\chi_B \rangle =0 $ for any $A,B \in B_n$ with $A \neq B$. It remains to show that $\| \chi_B\|^2_2> 0$ for all $B \in B_n$. Theorem 10 in \cite{Filmus16} shows that, for $B \in B_n$ of length $d$,
\[
\| \chi_B\|_2^2=c_B\|\chi_d\|_2^2
\]
where $\chi_d=\prod_{i=1}^d (x_{2i-1}-x_{2i})$ and $c_B= \prod_{i=1}^d \frac{(b_i - 2(i-1))(b_i - 2(i-1)-1)}{2} $. By the definition of $B \in B_n$ (see the definition of the {\em top} set under the Definition 4 in \cite{Filmus16}) , one can deduce that $b_i \geq 2i$ for all $i \in [d]$. Hence $c_B>0$ and it suffices to show that $\| \chi_d\|_2^2 >0.$ Let $a\in \Omega_n$ such that $a_i=(-1)^i$ for all $i \in [n]$. Since $\chi_d(a) \neq 0$ and $\Pr\{x =a\}>0$ under the uniform probability measure on $\Omega_n$, we conclude that $\|\chi_d\|_2^2 > 0$, completing the proof. 
\end{proof}
For the rest of the section, we equip $\mathbb{R}^n$ with the uniform probability measure on $\Omega_n$, which is exchangeable. Hence, many of the results in \cite{Filmus16} apply. Let $n \in \mathbb{N}$ be even, and let $(a_1,...,a_n) \in \mathbb{R}^n$. By Theorem~\ref{uniqueness}, there exists $f \in H_n$ such that
\[
f(x)=|\sum_{i=1}^n a_i x_i| \quad \text{for all } x \in \Omega_n.
\]
\begin{theorem}[Spectrum is supported on even indices] \label{eveness}
The spectrum of $f$ defined above is supported on the even indices, i.e., 
\[
f(x) = \sum_{\substack{B \in B_n\\ |B| \text{ even}}} \widehat{f}(B)\chi_B(x).
\]    
\end{theorem}
\begin{proof}
Note that $f(x)=f(-x)$ for all $x \in \Omega_n$. Looking at this fact from the Young-Fourier expansion of $f$, we have that
\[
\sum_{B \in B_n} \widehat f(B)\chi_B(x)= \sum_{B \in B_n}\widehat{f}(B)(-1)^{|B|}\chi_B(x), \quad \text{for all } x \in \Omega_n
\]
due to the fact that $\chi_B$ is a homogeneous multilinear polynomial of degree $|B|$. After rearranging the terms, we have that 
\[
\sum_{\substack{B \in B_n\\|B| \text{ is odd}}} \widehat{f}(B)\chi_B(x)=0,\quad \text{for all }x \in \Omega_n.
\]
Since the functions $\{\chi_B\}_{B\in B_n}$ are linearly independent by Theorem~\ref{uniqueness}, we know that $\widehat{f}(B)=0$ for all $B \in B_n$ with $|B|$ odd, completing the proof.     
\end{proof}

A fact used in the Kwapie\'n-Lata\l a-Oleszkiewicz's proof of the Szarek's inequality is a tightened Poincar\'e inequality for $f$ discussed above. We now derive an analogue on the middle slice. 
\begin{theorem}[A tightened Poincar\'e inequality] \label{Poincare}
 Let $n\in \mathbb{N}$ be even. Let $f\in H_n$ such that $f(x)=|\sum_{i=1}^n a_ix_i|$ for all $x \in \Omega_n$.  $\mathrm{Var}(f):=\mathbb{E}_{\mathrm{unif}\{\Omega_n\}}(f - \mathbb{E}_{\mathrm{unif}\{\Omega_n\}} f)^2.$ Then we have the tightened Poincar\'e inequality 
 \[
 \frac{4}{n}\mathrm{Var}(f) \leq \langle f,Lf \rangle.
 \]
\end{theorem}
\begin{proof}
\[
\langle f,Lf \rangle = \sum_{B \in B_n} \frac{2|B|(n+1-|B|)}{n(n-1)}\widehat f(B)^2 \|\chi_B\|_2^2 = \sum_{\text{even } |B| \geq 2 }\frac{2|B|(n+1-|B|)}{n(n-1)}\widehat f(B)^2 \|\chi_B\|_2^2,
\]
since $\widehat{f}(B)=0$ for $B$  with $|B|$ odd by Theorem~\ref{eveness}. Moreover,
\[
\mathrm{Var}(f) = \langle f - \mathbb{E}_{\mathrm{unif}\{\Omega_n\}} f, f - \mathbb{E}_{\mathrm{unif}\{\Omega_n\}} f \rangle = \langle f - \widehat f(\emptyset) , f - \widehat f(\emptyset) \rangle = \sum_{\text{even } |B| \geq 2 }\widehat f(B)^2 \|\chi_B\|_2^2.
\]
By \cite[Lemma 2]{Filmus16}, $|B| \leq \frac{n}{2}$. Hence we have $\frac{2|B|(n+1-|B|)}{n(n-1)} \geq \frac{4}{n}$ for even $|B| \geq 2$ by direct calculation, completing the proof. 
\end{proof}

Another fact used in the Kwapie\'n-Lata\l a-Oleszkiewicz's proof of the Szarek's inequality is a pointwise bound for $Lf$. We now derive an analogue on the middle slice. 
\begin{lemma}\label{lemma-Lf}
Let $n\in \mathbb{N}$ be even and $n \geq 4$. Let $f$ be the function discussed above. Then
\[
Lf(x) \leq \frac{2}{n-1} f(x), \, \quad \text{for all } x \in \Omega_n.
\]
\end{lemma}
\begin{proof}
Denote by $S_n(x)=\sum_{i=1}^n a_ix_i$. Note that for $x \in \Omega_n$,
\begin{align*}
    \sum_{i<j} S_n^{(i,j)}(x) &= \sum_{i<j} \big(S_n(x) + (a_i-a_j)(x_j-x_i)\bigr) \\
    & = \binom{n}{2} S_n(x) + \frac{1}{2}\sum_{i\neq j}(a_ix_j -a_ix_i-a_jx_j+a_jx_i) \\
    & = \binom{n}{2} S_n(x) + \frac{1}{2}\sum_{i\neq j}a_ix_j - \frac{1}{2}\sum_{i\neq j}a_ix_i- \frac{1}{2}\sum_{i\neq j}a_jx_j+ \frac{1}{2}\sum_{i\neq j}a_jx_i\\
    & =  \binom{n}{2} S_n(x) + \frac{1}{2}\sum_{i=1}^n (a_i(\sum_{j=1}^n x_j)-a_ix_i) - \frac{1}{2}(n-1)S_n(x) \\ &- \frac{1}{2}(n-1)S_n(x) + \frac{1}{2}\sum_{j=1}^n (a_j(\sum_{i=1}^n x_i)-a_jx_j) \\
    & = \binom{n}{2} S_n(x) - \frac{1}{2}S_n(x) - \frac{1}{2}(n-1)S_n(x)- \frac{1}{2}(n-1)S_n(x) -  \frac{1}{2}S_n(x), \quad\text{since}\sum_{i=1}^n x_i=0 \\
    & = (\binom{n}{2} -n )S_n(x).
\end{align*}
Hence, for $x \in \Omega_n$
\begin{align*}
    Lf(x) &= |S_n(x)| - \frac{1}{\binom{n}{2}}\sum_{i<j}|S_n^{(i,j)}(x)| \\
    & \leq |S_n(x)| - \frac{1}{\binom{n}{2}}|\sum_{i<j}S_n^{(i,j)}(x)| \\
    & = |S_n(x)| - \frac{1}{\binom{n}{2}}|(\binom{n}{2}-n)S_n(x)| \\
    & = \frac{2}{n-1}f(x).
\end{align*}
\end{proof}

\begin{proof}[Proof of Theorem~\ref{Szarek}]
The upper bound follows from  Jensen's inequality. For the lower bound, we proceed as follows. Let $f(x)$ be the unique multilinear harmonic polynomial over $\mathbb{R}^n$ such that $f(x)=|\sum_{i=1}^n a_ix_i|$ for all $x\in \Omega_n$. Then by Theorem~\ref{Poincare} and Lemma~\ref{lemma-Lf}, we have
\[
\frac{4}{n}\mathrm{Var}(f) \leq \langle f,Lf\rangle \leq  \frac{2}{n-1} \|f\|_2^2.
\]
Since $\mathrm{Var}(f) = \|f\|_2^2 - (\mathbb{E}f)^2$, we obtain
\[
\mathbb{E}f \geq \sqrt{\frac{n-2}{2(n-1)}} \|f\|_2,
\]
completing the proof. 
\end{proof}
\begin{remark}
It would be interesting to explore further which symmetry condition of the distribution of $(x_1,\ldots,x_n)$ leads to Khintchine-type inequalities, without assuming that the $x_i$ are independent.
\end{remark}

\subsection{Proof of Theorem~\ref{sidon}}

The lower bound \(c\ge 1.28\) was proved in \cite{CloningerSteinerberger17}. The key estimate, established in \cite[Lemma~3]{CloningerSteinerberger17}, is that for all \(m,n\in\mathbb N\),
\[
c\ge b_{n,m}-\frac{2}{m}-\frac{1}{m^2},
\]
where \(b_{n,m}\) is defined in terms of a discrete set \(B_{n,m}\); see \cite[Lemma~3]{CloningerSteinerberger17} for the definition of $B_{n,m}$ and further details. Using a large-scale computation, Cloninger and Steinerberger \cite[Section 3.3]{CloningerSteinerberger17} showed that for \(n=24\) and \(m=50\),
\[
b_{n,m}\ge 1.28+\frac{2}{m}+\frac{1}{m^2},
\]
which yields \(c\ge 1.28\). Their computation required about \(20000\) CPU hours on a high-performance server.

Our improvement comes from a small refinement of the error term. Namely, we replace the estimate above by
\[
c\ge b_{n,m}-\frac{2}{m}-\frac{1}{2m^2}.
\]
Using the same computed value of \(b_{24,50}\), this gives
\[
c\ge 1.2802.
\]

The argument is very simple: the only change is in the bound for \(\varepsilon\ast\varepsilon\). In \cite[Lemma~3]{CloningerSteinerberger17}, one uses
\[
\|\varepsilon\ast\varepsilon\|_{L^\infty(\mathbb R)}\le \frac{1}{m^2}.
\]
Since \(\varepsilon\) is supported on \(\left[-\tfrac14,\tfrac14\right]\), which has length \(\tfrac12\), we have
\[
\|\varepsilon\|_{L^1(\mathbb R)}
\le \frac12\,\|\varepsilon\|_{L^\infty(\mathbb R)}
\le \frac{1}{2m}.
\]
Therefore,
\[
\|\varepsilon\ast\varepsilon\|_{L^\infty(\mathbb R)}
\le \|\varepsilon\|_{L^\infty(\mathbb R)}\|\varepsilon\|_{L^1(\mathbb R)}
\le \frac1m\cdot\frac{1}{2m}
= \frac{1}{2m^2}.
\]
Substituting this into the proof of \cite[Lemma~3]{CloningerSteinerberger17} yields
\[
c\ge b_{n,m}-\frac{2}{m}-\frac{1}{2m^2},
\]
as claimed.

\section*{Acknowledgments}
 P.I. acknowledges partial support from the NSF CAREER grant DMS-2152401, a Simons Fellowship, and a Humboldt Research Fellowship for Experienced Researchers. 

\appendix
\section{Chat Links}\label{appendix:A}

\begin{itemize}
    \item Theorem~\ref{Gaussian}. Chat link with Grok 4.20(Beta). \\
\url{https://grok.com/share/c2hhcmQtNA_452b1841-e7b1-4695-a6aa-cf3210365900?rid=f9834ecf-f716-4a48-81c6-15aefa0dcba7} \\

    \item Upper bound in Theorem~\ref{lem:sqrt3}. Chat link with Grok 4 Expert. \\
    \url{https://grok.com/share/c2hhcmQtNA_84e20fd0-d81a-4809-9728-83d37e88cd5f?rid=f6483ca7-f772-4e51-a8e2-f62926b6dfab}\\

    \item Lower bound in Theorem~\ref{lem:sqrt3}. Chat link with Grok 4 Heavy. \\
    \url{https://grok.com/share/c2hhcmQtNA_826731f3-afbc-42a5-aa61-acca3ec1324a?rid=de4a202d-357d-4ece-9954-16c0dc5951d5}\\

    \item Theorem~\ref{Szarek}. Chat link with Grok 4.30(Beta).\\
    \url{https://grok.com/share/c2hhcmQtNA_de1dfee6-b59f-4fa8-a439-aa176d2308c6}\\

    \item Theorem~\ref{sidon}. Chat link with Grok 4 Heavy.\\
    \url{https://grok.com/share/c2hhcmQtNQ_f4d17f80-4582-4679-b931-06277fd4cfd4?rid=a60436ae-eaba-4638-a0fd-47b231f19cd0}

\end{itemize}

\bibliographystyle{abbrv}

\end{document}